\documentclass[11pt]{article}

\usepackage{enumerate}
\usepackage{amsmath}
\usepackage{amssymb,latexsym}
\usepackage{amsthm}
\usepackage{color}
\usepackage[normalem]{ulem}

\usepackage{graphicx}

\usepackage{hyperref}

\DeclareMathOperator{\C}{\mathcal{C}}
\DeclareMathOperator{\V}{\mathbb{V}}

\newtheorem{theorem}{Theorem}[section]
\newtheorem*{theorem*}{Theorem}

\newtheorem{lemma}[theorem]{Lemma}
\newtheorem{corollary}[theorem]{Corollary}
\newtheorem{definition}[theorem]{Definition}
\newtheorem{proposition}[theorem]{Proposition}
\newtheorem{result}[theorem]{Result}
\newtheorem{example}[theorem]{Example}

\newtheorem{remark}[theorem]{Remark}

\newcommand{\cS}{{\mathcal S}}

\newcommand{\cC}{{\mathcal C}}

\newcommand{\cL}{{\mathcal L}}

\newcommand{\F}{{\mathbb F}}

\newcommand{\la}{\langle}
\newcommand{\ra}{\rangle}

\newcommand{\PG}{\mathrm{PG}}

\newcommand{\G}{\mathrm{\Gamma L}}

\newcommand\qbin[3]{\left[\begin{matrix} #1 \\ #2 \end{matrix} \right]_{#3}}

\title{Generalising the scattered property of subspaces}
\author{Bence Csajb\'ok, Giuseppe Marino, Olga Polverino and Ferdinando Zullo \thanks{
The
research  was supported by
 the Italian National
Group for Algebraic and Geometric Structures and their Applications (GNSAGA
- INdAM). The first author was partially supported by the J\'anos Bolyai
Research Scholarship of the Hungarian Academy of Sciences and by OTKA grants PD 132463 and K 124950.  
The last two authors were supported by the project ``VALERE: VAnviteLli pEr la RicErca" of the University of Campania ``Luigi Vanvitelli''.}}

\begin{document}
\maketitle

\begin{abstract}
Let $V$ be an $r$-dimensional $\F_{q^n}$-vector space. We call an $\F_q$-subspace $U$ of $V$
$h$-scattered if $U$ meets the $h$-dimensional $\F_{q^n}$-subspaces of $V$ in $\F_q$-subspaces of dimension at most $h$. In 2000 Blokhuis and Lavrauw proved that $\dim_{\F_q} U \leq rn/2$ when $U$ is $1$-scattered. Subspaces attaining this bound have been investigated intensively because of their relations with projective two-weight codes and strongly regular graphs. MRD-codes with a maximum idealiser have also been linked to $rn/2$-dimensional $1$-scattered subspaces and to $n$-dimensional $(r-1)$-scattered subspaces.

In this paper we prove the upper bound $rn/(h+1)$ for the dimension of $h$-scattered subspaces, $h>1$, and construct examples with this dimension. We study their intersection numbers with hyperplanes, introduce a duality relation among them, and study the equivalence problem of the corresponding linear sets.
\end{abstract}

\section{Introduction}

Let $V(n,q)$ denote an $n$-dimensional $\F_q$-vector space.
A $t$-spread of $V(n,q)$ is a set $\cS$ of $t$-dimensional $\F_q$-subspaces such that each vector of $V(n,q)\setminus \{{\bf 0}\}$ is contained in exactly one element of $\cS$.
As shown by Segre in \cite{Segre}, a $t$-spread of $V(n,q)$ exists if and only if $t \mid n$.

Let $V$ be an $r$-dimensional $\F_{q^n}$-vector space and let $\cS$ be an $n$-spread of $V$, viewed as an $\F_q$-vector space.
An $\F_q$-subspace $U$ of $V$ is called \emph{scattered} w.r.t. $\cS$ if it meets every element of $\cS$ in an $\F_q$-subspace of dimension at most one, see \cite{BL2000}.
If we consider $V$ as an $rn$-dimensional $\F_q$-vector space, then it
is well-known that the one-dimensional $\F_{q^n}$-subspaces of $V$, viewed as $n$-dimensional $\F_q$-subspaces, form an $n$-spread of $V$. This spread is called the \emph{Desarguesian spread}.
In this paper scattered will always mean scattered w.r.t. the Desarguesian spread.
For such subspaces Blokhuis and Lavrauw showed that their dimension can be bounded by $rn/2$.
After a series of papers, it is now known that when $2 \mid rn$ then there always exist scattered subspaces of this dimension \cite{BBL2000, BGMP2015, BL2000, CSMPZ2016}. 

In this paper we introduce and study the following special class of scattered subspaces.

\begin{definition}
	Let $V$ be an $r$-dimensional $\F_{q^n}$-vector space.
	An $\F_q$-subspace $U$ of $V$ is called $h$-scattered, $0<h \leq r-1$, if $\la U \ra_{\F_{q^n}}=V$ and each $h$-dimensional $\F_{q^n}$-subspace of $V$ meets $U$ in an $\F_q$-subspace of dimension at most $h$. An $h$-scattered subspace of highest possible dimension is called a maximum $h$-scattered subspace.
\end{definition}

With this definition, the $1$-scattered subspaces are the scattered subspaces generating $V$ over $\F_{q^n}$. With $h=r$ the above definition would give the $n$-dimensional $\F_q$-subspaces of $V$ defining subgeometries of $\PG(V,\F_{q^n})$.
If $h=r-1$ and $\dim_{\F_q} U=n$, then $U$ defines a scattered $\F_q$-linear set with respect to hyperplanes, introduced in \cite[Definition 14]{ShVdV}.
A further generalisation of the concept of $h$-scattered subspaces can be found in the recent paper \cite{BaCsMT2020}.

In this paper we prove that for an $h$-scattered subspace $U$ of $V(r,q^n)$, if $U$ does not define a subgeometry, then

\begin{equation}\label{hscatbound}
\dim_{\F_q} U \leq \frac{rn}{h+1},
\end{equation}
cf.\ Theorem \ref{bound}. Clearly, $h$-scattered subspaces reaching bound \eqref{hscatbound} are maximum $h$-scattered. 
When $h+1 \mid r$ then our examples prove that maximum $h$-scattered subspaces have dimension $rn/(h+1)$, cf.\ Theorem \ref{dirsum}.
In Theorem \ref{maininter} we show that $h$-scattered subspaces of dimension $rn/(h+1)$ meet hyperplanes of $V(r,q^n)$ in $\F_q$-subspaces of dimension at least $rn/(h+1)-n$ and at most $rn/(h+1)-n+h$. Then we introduce a duality relation between maximum $h$-scattered subspaces of $V(r,q^n)$ reaching bound \eqref{hscatbound} and maximum $(n-h-2)$-scattered subspaces of $V(rn/(h+1)-r,q^n)$ reaching bound \eqref{hscatbound}, which allows us to give some constructions also when $h+1$ is not a divisor of $r$, cf. Theorem \ref{newex}.

Proposition \ref{scatole} shows us that $h$-scattered subspaces are special classes of $1$-scattered subspaces.
In \cite[Corollary 4.4]{ShVdV} the $(r-1)$-scattered subspaces of $V(r,q^n)$ attaining bound \eqref{hscatbound}, i.e. of dimension $n$, have been shown to be equivalent to MRD-codes of $\F_{q}^{n\times n}$ with minimum rank distance $n-r+1$ and with left or right idealiser isomorphic to $\F_{q^n}$. In Section \ref{sec:linearsets} we study the $\F_q$-linear set $L_U$  determined by an $h$-scattered subspace $U$.  In contrast to the case of $1$-scattered subspaces, it turns out that for any $h$-scattered $\F_q$-subspaces $U$ and $W$ of $V(r,q^n)$ with $h>1$, the corresponding linear sets $L_U$ and $L_W$ are $\mathrm{P}\Gamma\mathrm{L}(r,q^n)$-equivalent if and only if $U$ and $W$ are $\Gamma\mathrm{L}(r,q^n)$-equivalent, cf.\ Theorem \ref{simplicitythm}.
For $r>2$ this result extends \cite[Proposition 3.5]{ShVdV} regarding the equivalence between MRD-codes and maximum $(r-1)$-scattered subspaces attaining bound \eqref{hscatbound} into an equivalence between MRD-codes and the corresponding linear sets, see \cite[Remarks 4, 5]{ShVdV}.

\section{\texorpdfstring{{The maximum dimension of an $h$-scattered subspace}}{The maximum dimension of an h-scattered subspace}}

We start this section by the following result.

\begin{proposition}
		\label{scatole}
		For $h>1$ the $h$-scattered subspaces are also $i$-scattered for any $i<h$. In particular they are all $1$-scattered.
\end{proposition}
\begin{proof}
    Let $U$ be an $h$-scattered subspace of $V$. Suppose to the contrary that it is not $i$-scattered for some $i<h$. Therefore, there exists an $i$-dimensional $\F_{q^n}$-subspace $S$ such that $\dim_{\F_q}(S\cap U)\geq i+1$. As $\la U\ra_{\F_{q^n}}=V$, there exist $\mathbf{u}_1,\ldots,\mathbf{u}_{h-i} \in U$ such that $\dim_{\F_{q^n}}\langle S, \mathbf{u}_1,\ldots,\mathbf{u}_{h-i} \rangle_{\mathbb{F}_{q^n}}=h$. Then \[\dim_{\F_q}\left(U\cap \langle S, \mathbf{u}_1,\ldots,\mathbf{u}_{h-i} \rangle_{\mathbb{F}_{q^n}}\right) \geq (i+1) +(h-i)=h+1,\]
    a contradiction.
\end{proof}

In the proof of the main result of this section we will need the following lemma.

\begin{lemma}
	\label{lemmabound}
For any integer $i$ with $r \leq i \leq n$ in $V=V(r,q^n)$ there exists an $(r-1)$-scattered $\F_q$-subspace of dimension $i$.
\end{lemma}
\begin{proof}
	Fix an $\F_{q^n}$-basis of $V$, then the space $V$ can be seen as $\F_{q^n}^r$. Consider the $n$-dimensional $\F_q$-subspace $U=\{(x,x^q,\ldots,x^{q^{r-1}}) : x\in \F_{q^n}\}$ of $V$. Let $W$ be any $i$-dimensional $\F_q$-subspace of $U$.
	The intersection of $W$ with a hyperplane $[a_0,a_1,\ldots,a_{r-1}]$ of $V$ is
	\[\left\{(x,x^q,\ldots, x^{q^{r-1}}) : x\in \F_{q^n},\, \sum_{j=0}^{r-1} a_j x^{q^j}=0\right\} \cap W,\]
	which is clearly an $\F_q$-subspace of size at most $\deg \sum_{j=0}^{r-1} a_j x^{q^j}\leq q^{r-1}$. 
	If $\la W \ra_{\F_{q^n}}\neq V$ then there was a hyperplane of $V$ containing $W$, a contradiction,
	i.e. $W$ is an $(r-1)$-scattered $\F_q$-subspace of $V$. 
\end{proof}

For $h=1$, the following result was shown in \cite{BL2000}.

\begin{theorem}\label{bound}
	Let $V$ be an $r$-dimensional $\F_{q^n}$-vector space and $U$ an $h$-scattered $\F_q$-subspace of $V$. Then either
	\begin{itemize}
		\item $\dim_{\F_q} U=r$, $U$ defines a subgeometry of $\PG(V,\F_{q^n})$ and $U$ is  $(r-1)$-scattered, or
		\item $\dim_{\F_q} U \leq rn/(h+1)$.
	\end{itemize}
\end{theorem}
\begin{proof}
Let $k$ denote the dimension of $U$ over $\F_q$. Since $\langle U \rangle_{\F_{q^n}}=V$, we have $k\geq r$ and in case of equality $U$ defines a subgeometry of $\PG(V,\F_{q^n})$ which is clearly $(r-1)$-scattered.
From now on we may assume $k>r$.
First consider the case $h=r-1$. Fix an $\F_{q^n}$-basis in $V$ and for ${\bf x} \in V$ denote the $i$-th coordinate w.r.t. this basis by $x_i$. Consider the following set of $\F_q$-linear maps from $U$ to $\F_{q^n}$:
\[
\cC_U:=\left\{G_{a_0,\ldots,a_{r-1}}\colon {\bf x}\in U \mapsto \sum_{i=0}^{r-1}a_i x_i : a_i \in \F_{q^n}\right\}.
\]
First we show that the non-zero maps of $\cC_U$ have rank at least $k-r+1$.
Indeed, if $(a_0,\ldots,a_{r-1})\neq\mathbf{0}$, then $\mathbf{u} \in \ker G_{a_0,\ldots,a_{r-1}}$ if and only if $\sum_{i=0}^{r-1} a_iu_i=0$, i.e. $\ker G_{a_0,\ldots,a_{r-1}}= U \cap H$, where $H$ is the hyperplane $[a_0,a_1,\ldots,a_{r-1}]$ of $V$.
Since $U$ is $(r-1)$-scattered, it follows that $\dim_{\F_q} \ker G_{a_0,\ldots,a_{r-1}} \leq r-1$ and hence the rank of $G_{a_0,\ldots,a_{r-1}}$ is at least $k-r+1$. Next we show that any two maps of $\cC_U$ are different.
Suppose to the contrary $G_{a_0,\ldots,a_{r-1}}=G_{b_0,\ldots,b_{r-1}}$, then $G_{a_0-b_0,\ldots,a_{r-1}-b_{r-1}}$ is the zero map.
If $(a_1-b_1,\ldots,a_r-b_r)\neq \mathbf{0}$, then $U$ would be contained in the hyperplane
$[a_0-b_0,a_1-b_1,\ldots,a_{r-1}-b_{r-1}]$, a contradiction since $\la U \ra_{\F_{q^n}}=V$.
Hence, $|\mathcal{C}_U|=q^{nr}$.

Suppose to the contrary $k>n$.
The elements of $\cC_U$ form a $nr$-dimensional $\F_q$-subspace of $\mathrm{Hom}_{\F_q}(U,\F_{q^n})$ and the non-zero maps of $\cC_U$ have rank at least $k-r+1$.
By Result \ref{Dels} (Singleton-like bound) we get
$q^{rn}\leq q^{k(n-k+r)}$ and hence $(k-n)(k-r)\leq 0$, which contradicts $k>r$.

From now on, we will assume $1<h<r-1$, since the assertion has been proved in \cite{BL2000} for $h=1$.
\medskip

First we assume $n\geq h+1$. Then by Lemma \ref{lemmabound}, in $\F_{q^n}^h$ there exists an $(h-1)$-scattered $\F_q$-subspace $W$ of dimension $h+1$.

Let $G$ be an $\F_q$-linear transformation from $V$ to itself with $\ker G=U$.
Clearly, $\dim_{\F_q} \mathrm{Im} \,G = rn-k$.
For each $(\mathbf{u}_1,\ldots,\mathbf{u}_h) \in V^h$ consider the $\F_{q^n}$-linear map
\[ \tau_{\mathbf{u}_1,\ldots,\mathbf{u}_h} \colon (\lambda_1,\ldots,\lambda_h) \in W \mapsto \lambda_1 \mathbf{u}_1+\ldots+\lambda_h \mathbf{u}_h\in V.\]
Consider the following set of $\F_q$-linear maps $W \rightarrow \mathrm{Im}\, G$
\[ \C:=\{G \circ \tau_{\mathbf{u}_1,\ldots,\mathbf{u}_h} : (\mathbf{u}_1,\ldots,\mathbf{u}_h) \in V^h \}.\]
Our aim is to show that these maps are pairwise distinct and hence $|\cC|=q^{rnh}$.
Suppose $G \circ \tau_{\mathbf{u}_1,\ldots,\mathbf{u}_h} =G \circ \tau_{\mathbf{v}_1,\ldots,\mathbf{v}_h}$. It follows that
$ G \circ \tau_{\mathbf{u}_1-\mathbf{v}_1,\ldots,\mathbf{u}_h-\mathbf{v}_h}$
is the zero map, i.e.
	\begin{equation}
	\label{ker}
	\lambda_1 (\mathbf{u}_1-\mathbf{v}_1)+ \ldots + \lambda_h (\mathbf{u}_h-\mathbf{v}_h) \in \ker G=U \mbox{ for each } (\lambda_1,\ldots,\lambda_h) \in W.
	\end{equation}

	\noindent For $i\in \{1,\ldots,h\}$, put ${\mathbf{z}}_i=\mathbf{u}_i-\mathbf{v}_i$, let $T:=\langle \mathbf{z}_1,\ldots,\mathbf{z}_h\rangle_{q^n}$ and let $t=\dim_{q^n}T$. We want to show that $t=0$.
	If $t=h$, then by \eqref{ker}
	\[ \{ \lambda_1\mathbf{z}_1+\ldots+\lambda_h \mathbf{z}_h : (\lambda_1,\ldots,\lambda_h) \in W \}\subseteq T\cap U, \]
	hence $\dim_{\F_q} (T \cap U) \geq \dim_{\F_q} W=h+1$, which is not possible since $T$ is an $h$-dimensional $\F_{q^n}$-subspace of $V$ and $U$ is $h$-scattered.
	Hence $0\leq t<h$. Assume $t\geq 1$. Let $\Phi: \F_{q^n}^h\rightarrow T$ be the $\F_{q^n}$-linear map defined by the rule
	\[(\lambda_1,\ldots,\lambda_h)\mapsto \lambda_1\mathbf{z}_1+\ldots+\lambda_h \mathbf{z}_h\] and consider the map $\tau_{\mathbf{z}_1,\ldots,\mathbf{z}_h}$. Note that $\tau_{\mathbf{z}_1,\ldots,\mathbf{z}_h}$ is the restriction of $\Phi$ on the $\F_q$-vector subspace $W$ of $\F_{q^n}^h$. It can be easily seen that
	\begin{equation}
	\label{eq5}
	\dim_{\F_{q^n}}\ker \Phi =h-t,
	\end{equation}
	\begin{equation}
	\label{eq6}
	\ker \tau_{\mathbf{z}_1,\ldots,\mathbf{z}_h} =\ker \Phi\cap W,
	\end{equation}
	and by \eqref{ker}
	\begin{equation}
	\label{eq6.1}
	\mathrm{Im}\, \tau_{\mathbf{z}_1,\ldots,\mathbf{z}_h} \subseteq T\cap U.
	\end{equation}
	Since $t\geq 1$, by Proposition \ref{scatole} the $\F_q$-subspace $W$ is $(h-t)$-scattered in $\F_{q^n}^h$ and hence taking \eqref{eq5} and \eqref{eq6} into account we get $\dim_{\F_q}\ker \tau_{\mathbf{z}_1,\ldots,\mathbf{z}_h} \leq h-t$, which yields
	\begin{equation}
	\label{eq7}
	\dim_{\F_q}\mathrm{Im}\,  \tau_{\mathbf{z}_1,\ldots,\mathbf{z}_h} \geq t+1.
	\end{equation}
	By Proposition \ref{scatole} the $\F_q$-subspace $U$ is also a $t$-scattered subspace of $V$, thus by \eqref{eq6.1}
	\[
	\dim_{\F_q} \mathrm{Im}\, \tau_{\mathbf{z}_1,\ldots,\mathbf{z}_h} \leq \dim_{\F_q}(T\cap U)\leq t,
	\]
	contradicting \eqref{eq7}. It follows that $t=0$, i.e. $\mathbf{z}_i=0$ for each $i\in\{1,\ldots h\}$ and hence $|\C|=q^{rnh}$.
The trivial upper bound for the size of $\C$ is the size of $\F_q^{(h+1)\times(rn-k)}$, thus
	\[ q^{rnh}=|\C|\leq q^{(h+1)(rn-k)}, \]
	which implies
	\[ k \leq \frac{rn}{h+1}. \]
	
\medskip

Now assume $n < h+1$. By Proposition \ref{scatole} $U$ is $h'$-scattered with $h'=n-1$.
Since $h'<r-1$ and $n\geq h'+1$, we can argue as before and derive $k=\dim_{\F_q} U \leq rn/(h'+1)=r$, contradicting $k>r$.
\end{proof}

The previous proof can be adapted also for the $h=1$ case without introducing the subspace $W$, cf.\ \cite{PhDthesis}.

\medskip


The following result is a generalisation of \cite[Theorem 3.1]{BGMP2015}.

\begin{theorem}\label{sum}
Let $V=V_1\oplus \ldots \oplus V_t$ where $V_i=V(r_i,q^n)$ and $V=V(r,q^n)$. If $U_i$ is an $h_i$-scattered $\F_q$-subspace in $V_i$, then the $\F_q$-subspace $U=U_1\oplus \ldots \oplus U_t $ is $h$-scattered in $V$, with $h=\min\{h_1,\ldots,h_t\}$.
Also, if $U_i$ is $h$-scattered in $V_i$ and its dimension reaches bound \eqref{hscatbound}, then $U$ is $h$-scattered in $V$ and its 
dimension reaches bound \eqref{hscatbound}.
\end{theorem}
\begin{proof}
Clearly, it is enough to prove the assertion for $t=2$.

If $h=1$, the result easily follows from Proposition \ref{scatole} and from \cite[Theorem 3.1]{BGMP2015}; hence, we may assume $h=h_1\geq 2$.

By way of contradiction suppose that there exists an $h$-dimensional $\F_{q^n}$-subspace $W$ of $V$ such that
\begin{equation}
\label{eq1}
\dim_{\F_q}(W\cap U)\geq h+1.
\end{equation}
Clearly, $W$ cannot be contained in $V_1$ since $U_1$ is $h$-scattered in $V_1$.
Let $W_1:=W\cap V_1$ and $s:=\dim_{\F_{q^n}}W_1$. Then $s < h$ and by Proposition \ref{scatole}, the $\F_q$-subspace $U_1$ is $s$-scattered in $V_1$, thus $\dim_{\F_q} (U_1 \cap W_1) \leq s$. Denoting
$\la U_1, W\cap U\ra_{\F_{q}}$ by $\bar U_1$, the Grassmann formula and \eqref{eq1} yield
\begin{equation}\label{eq4}
\dim_{\F_q} \bar{U}_1-\dim_{\F_{q}} U_1 \geq h+1-s.
\end{equation}
Consider the subspace $T:=W+V_1$ of the quotient space $V/V_1\cong V_2$. Then $\dim_{\F_{q^n}}T=h-s$ and $T$ contains the $\F_q$-subspace
\[M:=\bar U_1+V_1.\]
Since $M$ is also contained in the $\F_q$-subspace $U+V_1=U_2+V_1$, then $M$ is $h_2$-scattered in $V/V_1$ and hence by $h-s \leq h \leq h_2$ and by Proposition \ref{scatole}, $M$ is also $(h-s)$-scattered in $V/V_1$.

On the other hand,
\[\dim_{\F_q}(M\cap T)=\dim_{\F_q} M=\dim_{\F_q} \bar U_1-\dim_{\F_q} (\bar U_1\cap V_1)\geq\]
\[\dim_{\F_q} \bar U_1 -\dim_{\F_q}(U\cap V_1)=\dim_{\F_q} \bar U_1 -\dim_{\F_q} U_1,\]
and hence, by \eqref{eq4},
\[\dim_{\F_q}(M\cap T)\geq h-s+1,\]
a contradiction.

The last part follows from $rn /(h+1)=\sum_{i=1}^{t} r_i n /(h+1)$.
\end{proof}

\bigskip

Constructions of maximum $1$-scattered $\F_q$-subspaces of $V(r,q^n)$ exist for all values of $q$, $r$ and $n$, provided $rn$ is even \cite{BBL2000,BGMP2015,BL2000,CSMPZ2016}. For $r=3$, $n\leq 5$ see \cite[Section 5]{BaCsMT2020}. Also, there are constructions of maximum $(r-1)$-scattered $\F_q$-subspaces arising from MRD-codes (explained later in Section \ref{sec:MRD}) for all values of $q$, $r$ and $n$, cf.\ \cite[Corollary 4.4]{ShVdV}. In particular, the so called Gabidulin codes produce Example \ref{Gab}. One can also prove directly that these are maximum $(r-1)$-scattered subspaces by the same arguments as in the proof of Lemma \ref{lemmabound}.

\begin{example}
\label{Gab}
	In $\F_{q^n}^r$, if $n\geq r$, then the $\F_q$-subspace
	\[\{(x,x^q,x^{q^2},\ldots, x^{q^{r-1}}) : x\in \F_{q^n}\}\]
	is maximum $(r-1)$-scattered of dimension $n$.
\end{example}

\begin{theorem}
\label{dirsum}
If $h+1$ divides $r$ and $n\geq h+1$, then in $V=V(r,q^n)$ there exist maximum $h$-scattered $\F_q$-subspaces of dimension $rn/(h+1)$.
\end{theorem}
\begin{proof}
Put $r=t(h+1)$ and consider $V=V_1 \oplus \ldots \oplus V_t$,
with $V_i$ an $\F_{q^n}$-subspace of $V$ with dimension $h+1$.
For each $i$ consider a maximum $h$-scattered $\F_q$-subspace $U_i$ in $V_i$ of dimension $n$ which exists because of Example \ref{Gab}. By Theorem \ref{sum}, $U_1\oplus \ldots \oplus U_t$
is an $h$-scattered $\F_q$-subspace of $V$ with dimension $tn=\frac{rn}{h+1}$.
\end{proof}

In Theorem \ref{dirsum} we exhibit examples of maximum $h$-scattered subspaces of $V=V(r,q^n)$ whenever $h+1$ divides $r$. In Section \ref{DD} we introduce a method to construct such subspaces also when $h+1$ does not divide $r$. To do this, we will need an upper bound on the dimension of intersections of hyperplanes of $V$ with a maximum $h$-scattered subspace of dimension $rn/(h+1)$.
The proof of the following theorem is developed in Section \ref{sec:inter}.

\begin{theorem}
	\label{maininter}
	If $U$ is a maximum $h$-scattered $\F_q$-subspace of a vector space $V(r,q^n)$ of dimension $rn/(h+1)$, then for any $(r-1)$-dimensional $\F_{q^n}$-subspace $W$ of $V(r,q^n)$ we have
	\[\frac{rn}{(h+1)}-n\leq \dim_{\F_q}(U \cap W) \leq \frac{rn}{(h+1)}-n+h.\]
\end{theorem}

\noindent The above theorem is a generalisation of \cite[Theorem 4.2]{BL2000} and the first part of its proof relies on the counting technique developed in \cite[Theorem 4.2]{BL2000}.

\section{\texorpdfstring{Delsarte dual of an $h$-scattered subspace}{Delsarte dual of an h-scattered subspace}}
\label{DD}

Let $U$ be a $k$-dimensional $\F_q$-subspace of a vector space $\Lambda=V(r,q^n)$, with $k>r$. By \cite[Theorems 1, 2]{LuPo2004} (see also \cite[Theorem 1]{LuPoPo2002}), there is an embedding of $\Lambda$ in $\V=V(k,q^n)$ with $\V=\Lambda \oplus \Gamma$ for some $(k-r)$-dimensional $\F_{q^n}$-subspace $\Gamma$ such that
$U=\la W,\Gamma\ra_{\F_{q}}\cap \Lambda$, where $W$ is a $k$-dimensional $\F_q$-subspace of $\V$, $\langle W\rangle_{\F_{q^n}}=\V$ and $W\cap \Gamma=\{{\bf 0}\}$.
Then the quotient space $\V/\Gamma$ is isomorphic to $\Lambda$ and under this isomorphism $U$ is the image of the $\F_q$-subspace $W+\Gamma$ of $\V /\Gamma$.

\medskip

Now, let $\beta'\colon W\times W\rightarrow\F_{q}$ be a non-degenerate reflexive sesquilinear form on $W$ with companion automorphism $\sigma'$. Then $\beta'$ can be extended to a non-degenerate reflexive sesquilinear form $\beta\colon \V\times\V\rightarrow\F_{q^n}$. Indeed if $\{{\bf u}_1,\ldots,{\bf u}_k\}$  is an $\F_q$-basis of $W$, since $\la W\ra_{\F_{q^n}}=\V$, for each ${\bf v}, {\bf w}\in\V$ we have \[\beta({\bf v}, {\bf w})=\sum_{i,j=1}^k a_ib_j^\sigma\beta'({\bf u}_i, {\bf u}_j),\] where ${\bf v}=\sum_{i=1}^ka_i{\bf u}_i$, ${\bf w}=\sum_{i=1}^kb_i{\bf u}_i$ and $\sigma$ is an automorphism of $\F_{q^n}$ such that $\sigma_{|\F_q}=\sigma'$.
Let $\perp$ and $\perp'$ be the orthogonal complement maps defined by $\beta$ and $\beta'$ on the lattice of $\F_{q^n}$-subspaces of $\V$ and of $\F_q$-subspaces of $W$, respectively.
For an $\F_q$-subspace $S$ of $W$ the $\F_{q^n}$-subspace $\la S \ra_{\F_{q^n}}$ of $\V$ will be denoted by $S^*$. In this case $(S^*)^{\perp}=(S^{\perp'})^*$.

\medskip

In this setting, we can prove the following preliminary result.

\begin{proposition}
	\label{prop:dual}
Let $W$, $\Lambda$, $\Gamma$, $\V$, $\perp$ and $\perp'$ be defined as above. If $U$ is a $k$-dimensional $\F_q$-subspace of $\Lambda$ with $k>r$ and

\begin{equation}
\tag{$\diamond$}\label{form:star}
\mbox{$\dim_{\F_q}(M\cap U)<k-1$ holds for each hyperplane $M$ of $\Lambda$,}
\end{equation}

then $W+\Gamma^\perp$ is a $k$-dimensional $\F_q$-subspace of the quotient space $\V/ \Gamma^\perp$. 
\end{proposition}
\begin{proof}
As described above, $U$ turns out to be isomorphic to the $\F_q$-subspace $W+\Gamma$ of the quotient space $\V/\Gamma$. 
By \eqref{form:star}, since each hyperplane of $\V/\Gamma$ is of form $H+\Gamma$ where $H$ is a hyperplane of $\V$ containing $\Gamma$, it follows that
\begin{small}
\begin{equation}\tag{$\diamond\diamond$}\label{form:starstar}
\mbox{$\dim_{\F_q}(H\cap W)<k-1$ for each hyperplane $H$ of $\V$ containing $\Gamma$. }
\end{equation}
\end{small}
To prove the assertion it is enough to prove
\[
W\cap \Gamma^\perp=\{{\bf 0}\}.
\]
Indeed, by way of contradiction, suppose that there exists a nonzero vector ${\bf v}\in W\cap\Gamma^\perp$. Then the $\F_{q^n}$-hyperplane $\langle {\bf v}\rangle_{\F_{q^n}}^\perp$ of $\V$ contains the subspace $\Gamma$ and meets $W$ in the $(k-1)$-dimensional $\F_q$-subspace $\langle {\bf v}\rangle_{\F_q}^{\perp'}$, which contradicts \eqref{form:starstar}.
\end{proof}

\begin{definition}
\label{deffff}	
\rm
Let $U$ be a $k$-dimensional $\F_q$-subspace of $\Lambda=V(r,q^n)$, with $k>r$ and such that  \eqref{form:star} is satisfied. Then the $k$-dimensional $\F_q$-subspace $W+\Gamma^{\perp}$ of the
quotient space $\V/\Gamma^{\perp}$ (cf.\ Proposition \ref{prop:dual}) will be denoted by $\bar U$ and we call it the \emph{Delsarte dual}{} of $U$ (w.r.t.\,$\perp$).
\end{definition}

The term Delsarte dual comes from the Delsarte dual operation acting on MRD-codes, as pointed out in Theorem \ref{th:termDD}.

\begin{theorem}
\label{thm:dual}
Let $U$ be a maximum $h$-scattered $\F_q$-subspace of a vector space $\Lambda=V(r,q^n)$ of dimension $rn/(h+1)$, with $n\geq h+3$.  Then the  $\F_q$-subspace $\bar U$ of $\V/\Gamma^\perp=V(rn/(h+1)-r,q^n)$ obtained by the procedure of Proposition \ref{prop:dual} is maximum $(n-h-2)$-scattered.
\end{theorem}
\begin{proof}
Put $k:=rn/(h+1)$. We first note that condition \eqref{form:star} is satisfied for $U$
since by Theorem \ref{maininter} the hyperplanes of $\Lambda$ meet $U$ in $\F_q$-subspaces of dimension at most $rn/(h+1)-n+h < k-1$. Also, $k>r$ holds since $n\geq h+3$.

Hence we can apply the procedure of Proposition \ref{prop:dual} to obtain the $\F_q$-subspace $\bar U=W+\Gamma^{\perp}$ of $\V/\Gamma^\perp$ of dimension $k$.

By way of contradiction, suppose that there exists an $(n-h-2)$-dimensional $\F_{q^n}$-subspace of $\V/\Gamma^\perp$, say $M$, such that
\begin{equation}\label{form2}
\dim_{\F_q}(M\cap \bar U)\geq n-h-1.
\end{equation}
Then $M=H+\Gamma^{\perp}$, for some $(n+r-h-2)$-dimensional $\F_{q^n}$-subspace $H$ of $\V$ containing $\Gamma^\perp$. For $H$, by \eqref{form2}, it follows that
\[
\dim_{\F_q}(H\cap W)=\dim_{\F_q}(M \cap \bar U)\geq n-h-1.
\]
Let $S$ be an $(n-h-1)$-dimensional $\F_q$-subspace of $W$ contained in $H$ and let $S^*:=\langle S\rangle_{\F_{q^n}}$. Then, $\dim_{\F_{q^n}}S^*=n-h-1$,
\begin{equation}\label{form3}
S^{\perp'}=W\cap (S^*)^\perp \mbox{\quad and \quad}   S^{\perp'}\subset (S^*)^\perp=\langle S^{\perp'}\rangle_{\F_{q^n}}.
\end{equation}
Since $S\subseteq H\cap W$ and $\Gamma^\perp\subset H$, we get $S^*\subset H$ and $H^\perp\subset \Gamma$, i.e.
\begin{equation}\label{form4}
H^\perp\subseteq \Gamma\cap (S^*)^\perp.
\end{equation}
From \eqref{form4} it follows that
\[
\dim_{\F_{q^n}}\left(\Gamma\cap (S^*)^\perp\right)\geq \dim_{\F_{q^n}}H^\perp=k-(n+r-h-2).
\]
This implies that
\[
\dim_{\F_{q^n}}\la\Gamma,(S^*)^\perp\ra_{\F_{q^n}}= \dim_{\F_{q^n}} \Gamma + \dim_{\F_{q^n}} (S^*)^{\perp} - \dim_{\F_{q^n}}\left(\Gamma\cap (S^*)^\perp\right) \leq k-1
\]
and hence $\la \Gamma, (S^*)^\perp \ra_{\F_{q^n}}$ is contained in a hyperplane $T$ of $\V$ containing $\Gamma$. Also, $\dim_{\F_q}(S^{\perp'})=\dim_{\F_q} W-\dim_{\F_q} S=k-(n-h-1)$ and, by \eqref{form3}, we get
\[S^{\perp'}= W\cap (S^*)^\perp\subseteq W\cap T.\]
Then $\hat T:= T \cap \Lambda$ is a hyperplane of $\Lambda$ and, by recalling $U=\la W,\Gamma \ra_{\F_{q}}\cap \Lambda$,
\[\dim_{\F_q}(\hat T \cap U)=\dim_{\F_q}(T \cap W)\geq \dim_{\F_q}(S^{\perp'})=k-n+h+1,\]
contradicting Theorem \ref{maininter}.

\end{proof}

In case of $h=r-1$, Theorem \ref{thm:dual} follows from \cite{ShVdV} and from the theory of MRD codes. 
Our theorem generalises this result to each value of $h$ by using a geometric approach.

\begin{corollary}
	Starting from a maximum $(r-1)$-scattered $\F_q$-subspace $U$ of $V(r,q^n)$ of dimension $n$, $n\geq r+2$, the $\F_q$-subspace $\bar U$ (cf.\ Definition \ref{deffff}) is a maximum $(n-r-1)$-scattered $\F_q$-subspace of $V(n-r,q^n)$ of dimension $n$.
\end{corollary}

\begin{corollary}
\label{maxex2}
	Starting from a maximum $1$-scattered $\F_q$-subspace $U$ of $V(r,q^n)$, $rn$ even, $n\geq 4$, $\bar U$ (cf.\ Definition \ref{deffff}) is a maximum $(n-3)$-scattered $\F_q$-subspace of $V(r(n-2)/2,q^n)$ whose dimension attains bound \eqref{hscatbound}.
	\qed
\end{corollary}

\begin{theorem}\label{newex}
If $n\geq 4$ is even and $r\geq 3$ is odd, then there exist maximum $(n-3)$-scattered $\F_q$-subspaces of $V(r(n-2)/2,q^n)$ which cannot be obtained from the direct sum construction of Theorem \ref{dirsum}.
\end{theorem}
\begin{proof}
	By \cite{BBL2000, BGMP2015, BL2000, CSMPZ2016} it is always possible to construct maximum $1$-scattered $\F_q$-subspaces of $V(r,q^n)$. Then the result follows from Corollary \ref{maxex2} and from the fact that in this case $n-2$ does not divide $r(n-2)/2$.
\end{proof}

\begin{remark}
{\rm
The Delsarte dual of an $\F_q$-subspace does not depend on the choice of the non-degenerate reflexive sesquilinear form on $W$.

Indeed, fix an $\F_{q}$-basis $B$ of $W$, since $\la W\ra_{\F_{q^n}}=\V$, we can see $W$ as $\F_q^k$ and $\V$ as $\F_{q^n}^k$. Let $\beta_1'$ and $\beta_2'$ be two non-degenerate reflexive sesquilinear forms on $\F_q^k$.
Then, with respect to the basis $B$, the forms $\beta_1'$ and $\beta_2'$ are defined by the following rules:
\[\beta_i'(({\bf x}, {\bf y}))={\bf x}G_i{\bf y}_t^{\rho_i} \quad \footnote{Here ${\bf y}_t$ denotes the transpose of the vector ${\bf y}$. },\]
where $G_i\in\mathrm{GL}(k,q)$ and $\rho_i$ is an automorphism of $\F_{q}$ such that $\rho_i^2=\mathrm{id}$ and $(G_i^{\rho_i})^t=G_i$, for $i\in\{1,2\}$.
Now let $\beta_1$ and $\beta_2$ be their extensions over $\F_{q^n}^k$ defined by the rules
\[\beta_i(({\bf x}, {\bf y}))={\bf x}G_i{\bf y}_t^{\rho_i},\]
and let $\perp_1$ and $\perp_2$ be the orthogonal complement maps defined by $\beta_1$ and $\beta_2$ on the lattice of $\F_{q^n}$-subspaces of $\F_{q^n}^k$, respectively.

Again w.r.t. the basis $B$, the $\F_{q^n}$-subspace $\Gamma$ described at the beginning of this section can be seen as a $(k-r)$-dimensional subspace of $\F_{q^n}^k$. Then, for $i\in\{1,2\}$ we have
\[\Gamma^{\perp_i}=\{{\bf x} : {\bf x}\,G_i\,{\bf y}_t^{\rho_i}=0\ \ \ \forall\, {\bf y}\in\Gamma\}.\]
Straightforward computations show that the invertible semilinear map
\[\varphi\colon  {\bf x}\in\F_{q^n}^k\mapsto {\bf x}^{\rho_2^{-1}\rho_1}G_2^{\rho_2^{-1}\rho_1}G_1^{-1}\in\F_{q^n}^k,\]
leaves $W$ invariant and maps $\Gamma^{\perp_2}$ to $\Gamma^{\perp_1}$. Then $\varphi$ maps $W+\Gamma^{\perp_2}$ to $W+\Gamma^{\perp_1}$, i.e. $\varphi$ maps the Delsarte dual of $U$ calculated w.r.t $\beta_2$ to the Delsarte dual of $U$ calculated w.r.t. $\beta_1$. See also \cite[Section 2]{Ravagnani}  and \cite[Section 6.2]{sheekey_newest_preprint}.
}
\end{remark}

\section{\texorpdfstring{Linear sets defined by $h$-scattered subspaces}{Linear sets defined by h-scattered subspaces}}
\label{sec:linearsets}

Let $V$ be an $r$-dimensional $\F_{q^n}$-vector space. A point set $L$ of $\Lambda=\PG(V,\F_{q^n})\allowbreak=\PG(r-1,q^n)$ is said to be an \emph{$\F_q$-linear set} of $\Lambda$ of rank $k$ if it is defined by the non-zero vectors of a $k$-dimensional $\F_q$-vector subspace $U$ of $V$, i.e.
\[L=L_U:=\{\la {\bf u} \ra_{\mathbb{F}_{q^n}} : {\bf u}\in U\setminus \{{\bf 0} \}\}.\]

One of the most natural questions about linear sets is their equivalence. Two linear sets $L_U$ and $L_W$ of $\PG(r-1,q^n)$ are said to be \emph{$\mathrm{P\Gamma L}$-equivalent} (or simply \emph{equivalent}) if there is an element $\varphi$ in $\mathrm{P\Gamma L}(r,q^n)$ such that $L_U^{\varphi} = L_W$. In the applications it is crucial to have methods to decide whether two linear sets are equivalent or not.  This can be a difficult problem and some results in this direction can be found in \cite{CsMP2, CsMP, CsZ2016}.
For $f\in \mathrm{\Gamma L}(r,q^n)$ we have $L_{U^{f}}=L_{U}^{\varphi_f}$, where $\varphi_f$ denotes the collineation of $\PG(V,\F_{q^n})$ induced by $f$. It follows that if $U$ and $W$ are $\F_q$-subspaces of $V$ belonging to the same orbit of
$\mathrm{\Gamma L}(r,q^n)$, then $L_U$ and $L_W$ are equivalent.
The above condition is only sufficient but not necessary to obtain equivalent linear sets.
This follows also from the fact that $\F_q$-subspaces of $V$ with different dimensions can define the same linear set, for example $\F_q$-linear sets of $\PG(r-1,q^n)$ of rank $k\geq rn-n+1$ are all the same: they coincide with $\PG(r-1,q^n)$.
Also, in \cite{CsMP, CsZ2016} for $r=2$ it was pointed out that there exist maximum $1$-scattered $\F_q$-subspaces of $V$ on different orbits of $\Gamma \mathrm{L}(2,q^n)$ defining $\mathrm{P\Gamma L}$-equivalent linear sets of $\PG(1,q^n)$.
It is then natural to ask for which linear sets can we translate the question of $\mathrm{P\Gamma L}$-equivalence into the question of $\G$-equivalence of the defining subspaces.
For further details on linear sets see \cite{Lavrauw,LVdV2015,Polverino}.

In this section we study the equivalence issue of $\F_q$-linear sets defined by $h$-scattered linear sets for $h\geq 2$.

\begin{definition}
	If $U$ is a (maximum) $h$-scattered $\F_q$-subspace of $V(r,q^n)$, then the $\F_q$-linear set $L_U$ of $\PG(r-1,q^n)$ is called (maximum) $h$-scattered.
\end{definition}

The $(r-1)$-scattered $\F_q$-linear sets of rank $n$ were defined also in \cite[Definition 14]{ShVdV} and following the authors of \cite{ShVdV}, we will call these $\F_q$-linear sets \emph{maximum scattered with respect to hyperplanes}. Also, we will call $2$-scattered $\F_q$-linear sets (of any rank) \emph{scattered with respect to lines}.

\begin{proposition}[{\cite[pg. 3 Eq. (6) and Lemma 2.1]{BoPol}}]
\label{BP0}
Let $V$ be a two-dimensional vector space over $\F_{q^n}$.
\begin{enumerate}
	\item If $U$ is an $\F_q$-subspace of $V$ with $|L_U|=q+1$, then $U$ has dimension $2$ over $\F_q$.
	\item Let $U$ and $W$ be two $\F_q$-subspaces of $V$ with $L_U=L_{W}$ of size $q+1$.
	If $U \cap W\neq \{\mathbf{0}\}$, then $U=W$.
\end{enumerate}
\end{proposition}

\begin{proposition}
\label{dim}
If $L_U$ is a scattered $\F_q$-linear set with respect to lines of $\PG(r-1,q^n)=\PG(V,\F_{q^n})$, then its rank is uniquely defined, i.e. for each $\F_q$-subspace $W$ of $V$ if $L_W=L_U$, then $\dim_{\F_q} W= \dim_{\F_{q}} U$.
\end{proposition}
\begin{proof}
Let $W$ be an $\F_q$-subspace of $V$ such that $L_U=L_W$ and put $k=\dim_{\F_q} U$.
Since $U$ is a $1$-scattered $\F_q$-subspace (cf.\ Proposition \ref{scatole}), $|L_U|=|L_W|=(q^k-1)/(q-1)$. It follows that $\dim_{\F_q} W \geq k$.
Suppose that $\dim_{\F_q} W \geq k+1$, then there exists at least one point $P=\la {\bf x}\ra_{\F_{q^n}} \in L_W$ such that $\dim_{\F_{q}} (W \cap \la {\bf x} \ra_{\F_{q^n}})\geq 2$.
Let $Q=\la {\bf y}\ra_{\F_{q^n}} \in L_U=L_W$ be a point different from $P$, then
$\la {\bf x},{\bf y}\ra_{\F_{q^n}}\cap W$ has dimension at least $3$ but the
linear set defined by $\la {\bf x},{\bf y}\ra_{\F_{q^n}}\cap W$ is $L_W \cap \la P,Q\ra$, thus it has size $q+1$, contradicting part 1 of Proposition \ref{BP0}.
\end{proof}

\begin{lemma}
\label{lem:simple}
Let $L_U$ be a scattered $\F_q$-linear set with respect to lines in $\PG(r-1,q^n)$.
If $L_U=L_W$ for some $\F_q$-subspace $W$, then $U=\lambda W$ for some $\lambda\in \F_{q^n}^*$.
\end{lemma}
\begin{proof}
By Proposition \ref{dim}, we have $\dim_{\F_q} W=\dim_{\F_q} U$ and hence, since $U$ is $1$-scattered, also
$W$ is $1$-scattered.
Let $P\in L_U$ with $P=\langle \mathbf{u}\rangle_{\F_{q^n}}$, then for some $\lambda \in \F_{q^n}^*$ we have $\mathbf{u} \in U \cap \lambda W$. Put $W':=\lambda W$ and note that $L_{W}=L_{W'}$.
Our aim is to prove $W' \subseteq U$. Since $U$ and $W'$ are $1$-scattered, we have
$\la {\bf u} \ra_{\F_{q^n}} \cap U=\la {\bf u} \ra_{\F_{q^n}} \cap W'=\la {\bf u}\ra_{\mathbb{F}_q}$.

What is left, is to show for each ${\bf w} \in W'\setminus \la {\bf u}\ra_{\F_{q^n}}$ that ${\bf w}\in U$. To do this, consider the point
$Q=\langle \mathbf{w}\rangle_{\F_{q^n}}\in L_{W'}=L_U$ and the line $\la P, Q\ra$ which meets $L_U$ in  $q+1$ points.
By part 1 of Proposition \ref{BP0}, the $\F_q$-subspace $\left(\langle \mathbf{u},\mathbf{w}\rangle_{\F_{q^n}}\cap U\right)$ has dimension $2$.
Since $\left(\langle \mathbf{u},\mathbf{w}\rangle_{\F_{q^n}}\cap U\right)\cap\left(\langle \mathbf{u},\mathbf{w}\rangle_{\F_{q^n}}\cap W'\right)\neq \{\mathbf{0}\}$, by part 2 of Proposition \ref{BP0} we get
\[\langle \mathbf{u},\mathbf{w}\rangle_{\F_{q^n}}\cap U=\langle \mathbf{u},\mathbf{w}\rangle_{\F_{q^n}}\cap W'=\langle \mathbf{u},\mathbf{w}\rangle_{\F_{q}}.\]
Hence the assertion follows.
\end{proof}

\begin{theorem}
	\label{simplicitythm}
	Consider two $h$-scattered linear sets $L_U$ and $L_W$ of $V(r,q^n)$ with $h\geq 2$. They are
	$\mathrm{P}\Gamma\mathrm{L}(r,q^n)$-equivalent if and only if $U$ and $W$ are $\Gamma\mathrm{L}(r,q^n)$-equivalent.
\end{theorem}
\begin{proof}
	The if part is trivial. To prove the only if part assume that there exists $f \in \G(r,q^n)$ such that $L_U^{\varphi_f}=L_W$, where $\varphi_f$ is the collineation induced by $f$. Since $L_U^{\varphi_f}=L_{U^f}$, by Proposition \ref{scatole} and Lemma \ref{lem:simple}, there exists $\lambda \in \F_{q^n}^*$ such that $\lambda U^f = W$ and hence $U$ and $W$ lie on the same orbit of $\G(r,q^n)$.
\end{proof}

\subsection{Scattered linear sets with respect to hyperplanes and MRD-codes}
\label{sec:MRD}

A \emph{rank distance (or RD) code} $\cC$ of $\F_q^{n\times m}$, $n \leq m$, can be considered as a subset of $\mathrm{Hom}_{\F_q}(U,V)$, where $\dim_{\F_{q}} U = m$ and $\dim_{\F_{q}} V = n$,
with \emph{rank distance} defined as $d(f,g):=\mathrm{rk}(f-g)$.
The minimum distance of $\cC$ is $d:=\min \{d(f,g) : f,g\in \cC, f\neq g\}$.

\begin{result}[\cite{Delsarte}]
\label{Dels}
	If $\cC$ is a rank distance code of $\F_q^{n\times m}$, $n \leq m$, with minimum distance $d$, then
\begin{equation}
\label{MRDbound}
	|\cC|\leq q^{m(n-d+1)}.
\end{equation}
\end{result}

Rank distance codes for which \eqref{MRDbound} holds with equality are called \emph{maximum rank distance (or MRD) codes}.

\medskip

From now on, we will only consider $\F_q$-linear MRD-codes of $\F_q^{n\times n}$, i.e. those which can be identified with $\F_q$-subspaces of $\mathrm{End}_{\F_q}(\F_{q^n})$.
Since $\mathrm{End}_{\F_q}(\F_{q^n})$ is isomorphic to the ring of $q$-polynomials over $\F_{q^n}$  modulo $x^{q^n}-x$, denoted by $\cL_{n,q}$, with addition and composition as operations, we will consider $\cC$ as an $\F_q$-subspace of $\cL_{n,q}$.
Given two $\F_q$-linear MRD codes, $\cC_1$ and $\cC_2$, they are equivalent if and only if there exist $\varphi_1$, $\varphi_2\in \cL_{n,q}$ permuting $\F_{q^n}$ and $\rho\in \mathrm{Aut}(\F_q)$ such that
\[ \varphi_1\circ f^\rho \circ \varphi_2 \in \cC_2 \text{ for all }f\in \cC_1,\]
where $\circ$ stands for the composition of maps and $f^\rho(x)= \sum a_i^\rho x^{q^i}$ for $f(x)=\sum a_i x^{q^i}$.
For a rank distance code $\cC$ given by a set of linearized polynomials, its left and right idealisers can be written as:
\[L(\cC)= \{ \varphi \in \cL_{n,q} : \varphi \circ f \in \cC \text{ for all }f\in \cC \},\]
\[R(\cC)= \{ \varphi \in \cL_{n,q} : f \circ \varphi \in \cC \text{ for all }f\in \cC \}.\]

By \cite[Section 2.7]{Lunardon2017} and \cite{ShVdV} the next result follows. We give a proof of the first part for the sake of completeness.

\begin{result}
	\label{cod}
	$\cC$ is an $\F_q$-linear MRD-code of $\cL_{n,q}$ with minimum distance $n-r+1$ and with left-idealiser isomorphic to $\F_{q^n}$ if and only if up to equivalence
	\[\cC=\la f_1(x),\ldots,f_r(x)\ra_{\F_{q^n}}\]
	for some $f_1,f_2,\ldots,f_r \in \cL_{n,q}$	and the $\F_q$-subspace
	\[U_{\cC}=\{(f_1(x),\ldots,f_r(x)) : x\in \F_{q^n}\}\]
	is a maximum $(r-1)$-scattered $\F_q$-subspace of $\F_{q^n}^r$.
\end{result}
\begin{proof}
	Let $T= \{\omega_a : a \in \F_{q^n}\}$, where for each $a\in \F_{q^n}$, $\omega_a(x)=ax \in \cL_{n,q}$ and let $L$ denote the left-idealiser of $\cC$. Since $T$ and $L$ are Singer cyclic subgroups of $\mathrm{GL}(\F_{q^n},\F_q)$ and any two such groups are conjugate (cf.\  \cite[pg. 187]{Huppert}) it follows that there exists an invertible $q$-polynomial $g$ such that $g\circ L \circ g^{-1}=T$. Then for each $h\in \cC':=g^{-1}\circ \cC$ it holds that $\omega_a \circ h \in \cC'$ for each $a\in \F_{q^n}$, which proves the first statement. For the second part see \cite[Corollary 4.4]{ShVdV}.
\end{proof}

\begin{remark}
The \emph{adjoint} of a $q$-polynomial $f(x)=\sum_{i=0}^{n-1}a_i x^{q^i}$, with respect to the bilinear form $\langle x,y\rangle:=\mathrm{Tr}_{q^n/q}(xy)$ (\footnote{Where $\mathrm{Tr}_{q^n/q}(x)=x+x^q+\ldots+x^{q^{n-1}}$ denotes the $\F_{q^n} \rightarrow \F_q$ trace function.}), is given by
\[\hat{f}(x):=\sum_{i=0}^{n-1}a_{i}^{q^{n-i}} x^{q^{n-i}}.\]
If $\cC$ is a rank distance code given by $q$-polynomials, then the \emph{adjoint code} $\cC^\top$ of $\cC$ is $\{\hat{f} : f\in\cC\}$. The code $\cC$ is an MRD if and only if $\cC^\top$ is an MRD and also $L(\cC) \cong R(\cC^\top)$, $R(\cC) \cong L(\cC^\top)$. Thus Result \ref{cod} can be translated also to codes with right-idealiser isomorphic to $\F_{q^n}$.
\end{remark}

The next result follows from \cite[Proposition 3.5]{ShVdV}.

\begin{result}
\label{resvdvs}
Let $\cC$ and $\cC'$ be two $\F_q$-linear MRD-codes of $\cL_{n,q}$ with minimum distance $n-r+1$ and with left-idealisers isomorphic to $\F_{q^n}$.
Then $U_{\cC}$ and $U_{\cC'}$ are $\G(r,q^n)$-equivalent if and only if $\cC$ and $\cC'$ are equivalent.
\end{result}

By Theorem \ref{simplicitythm}, for $r>2$ we can extend Result \ref{resvdvs} in the following way.

\begin{theorem}
Let $\cC$ and $\cC'$ be two $\F_q$-linear MRD-codes of $\cL_{n,q}$ with minimum distance $n-r+1$, $r>2$, and with left-idealisers isomorphic to $\F_{q^n}$.
Then the linear sets $L_{U_{\cC}}$ and $L_{U_{\cC'}}$ are $\mathrm{P}\Gamma\mathrm{L}(r,q^n)$-equivalent
if and only if $\cC$ and $\cC'$ are equivalent.
\end{theorem}

In the following we motivate why we used the term ``Delsarte dual" 
in Definition \ref{deffff}.
In particular, we prove that the duality of Section \ref{DD} corresponds to the Delsarte duality on MRD-codes when $(r-1)$-scattered $\F_q$-subspaces of $\F_{q^n}^r$ are considered.

First recall that in terms of linearized polynomials, the Delsarte dual of a rank distance code $\cC$ of $\cL_{n,q}$ introduced in \cite{Delsarte} and in \cite{Gabidulin} can be interpreted as follows
\[ \cC^\perp=\{f \in \cL_{n,q} : b(f,g)=0 \,\,\, \forall g \in \cC \}, \]
where $b(f,g)=\mathrm{Tr}_{q^n/q}\left(\sum_{i=0}^{n-1} a_ib_i\right)$ for $f(x)=\sum_{i=0}^{n-1} a_ix^{q^i}$ and $g(x)=\sum_{i=0}^{n-1} b_i x^{q^i} \in \mathcal{L}_{n,q}$.

\begin{remark}
\label{rem:termDD}{\rm
Let $\cC$ be an $\F_q$-linear MRD-code of $\cL_{n,q}$ with minimum distance $n-r+1$ and with left-idealiser isomorphic to $\F_{q^n}$.
By Result \ref{cod} and by \cite[Theorem 2.2]{CsMPZh}, there exist $r$ distinct integers in $\{0,\ldots,n-1\}$ such that, up to equivalence,
\[ \cC=\langle h_0(x),\ldots, h_{r-1}(x) \rangle_{\F_{q^n}}, \]
where 
\begin{equation}\label{eq:hi}
h_i(x)=x^{q^{t_i}}+\sum_{j \notin\{t_0,\ldots,t_{r-1}\}} g_{i,j}x^{q^j}\end{equation} 
and $g_{i,j} \in \F_{q^n}$. 

\noindent Also, let $\{s_0,s_1,\ldots,s_{n-r-1}\}:= \{0,\ldots,n-1\} \setminus \{t_0,\ldots,t_{r-1}\}$. Then it is easy to see that the Delsarte dual of $\cC$ is
\[\cC^{\perp}=\langle h_0'(x),\ldots, h_{n-r-1}'(x) \rangle_{\F_{q^n}},\]
where
\begin{equation}\label{eq:h'i} h_i'(x)=x^{q^{s_i}}-\sum_{j\in \{t_0,\ldots,t_{r-1}\}}g_{j,s_i}x^{q^j}.\end{equation}}
\end{remark}

\begin{theorem}\label{th:termDD}
Let $\cC$ be an $\F_q$-linear MRD-code of $\cL_{n,q}$ with minimum distance $n-r+1$ and with left-idealiser isomorphic to $\F_{q^n}$. Then there exist $h_0(x),\ldots, h_{r-1}(x),h_0'(x),\ldots, h_{n-r-1}'(x) \in \mathcal{L}_{n,q}$ such that, up to equivalence,
\begin{itemize}
    \item $\cC=\langle h_0(x),\ldots, h_{r-1}(x) \rangle_{\F_{q^n}}$,
    \item $\cC^{\perp}=\langle h_0'(x),\ldots, h_{n-r-1}'(x) \rangle_{\F_{q^n}}$,
    \item the Delsarte dual of $U_{\mathcal{C}}=\{ (h_0(x),\ldots, h_{r-1}(x)) : x \in \F_{q^n} \}$ is the $\F_q$-subspace $U_{\mathcal{C}^\perp}=\{ (h_0'(x),\ldots, h_{n-r-1}'(x)) : x \in \F_{q^n} \}$.
\end{itemize}    
\end{theorem}
\begin{proof}
By Remark \ref{rem:termDD}, up to equivalence, $\cC=\langle h_0(x),\ldots, h_{r-1}(x) \rangle_{\F_{q^n}}$, for some $h_0(x),\ldots,h_{r-1}(x)$ as in \eqref{eq:hi}, and $\cC^{\perp}=\langle h_0'(x),\ldots, h_{n-r-1}'(x) \rangle_{\F_{q^n}}$, for some $h_0'(x),\ldots,h_{n-r-1}'(x)$ as in \eqref{eq:h'i}.
Note that, since $\cC$ is an MRD-code, the linearized polynomials $h_0(x),\ldots,$ $h_{r-1}(x)$ have no common roots other than $0$ since otherwise the code would not contain invertible maps, see e.g. \cite[Lemma 2.1]{LTZ2}.
Our aim is to show that applying the duality introduced in Section \ref{DD} to $U_{\cC}=\{(h_0(x),\ldots,h_{r-1}(x)) : x \in \F_{q^n}\}$ we get the $\F_q$-subspace $U_{\cC^\perp}=\{(h_0'(x),\ldots, h_{n-r-1}'(x)) : x \in \F_{q^n}\}$.
By Result \ref{cod} we have that $U_{\cC}$ is a maximum $(r-1)$-scattered $\F_q$-subspace of $\F_{q^n}^r$.
If $n>r$, i.e. $\cC$ has minimum distance greater than one, we can embed $\Lambda=\langle U_{\cC} \rangle_{\F_{q^n}}$ in $\F_{q^n}^n$ in such a way that
\[ \Lambda=\left\{ (x_0,x_1,\ldots,x_r,\ldots,x_{n-1}) \in \F_{q^n}^n : x_j=0\,\,\, j \notin \{t_0,\ldots,t_{r-1}\} \right\},\]
and hence the vector $(h_0(x),\ldots, h_{r-1}(x))$ of $U_{\cC}$ is extended to the vector $(a_0,a_1,\ldots,a_{n-1})$ of $\F_{q^n}^n$ as follows
\[ a_i=\left\{ \begin{array}{ll} h_i(x) & \text{if}\,\, i \in \{t_0,\ldots,t_{r-1}\},\\ 0 & \text{otherwise}. \end{array} \right. \]
Let $\Gamma$ be the $\F_{q^n}$-subspace of $\F_{q^n}^n$ of dimension $n-r$ represented by the equations
\[ \Gamma : \left\{ \begin{array}{lll} x_{t_0}=-\displaystyle\sum_{j \notin\{t_0,\ldots,t_{r-1}\}} g_{0,j} x_j\\ \vdots \\ x_{t_{r-1}}=-\displaystyle\sum_{j \notin\{t_0,\ldots,t_{r-1}\}} g_{r-1,j} x_j \end{array} \right. \]
and let $W=\{(x,x^q,\ldots,x^{q^{n-1}}) : x \in \F_{q^n}\}$.
It can be seen that $\Gamma \cap W= \{{\bf 0}\}$, otherwise the polynomials $h_0(x),\ldots, h_{r-1}(x)$ would have a common root.
Also
\[ U_{\cC}= \langle W,\Gamma \rangle_{\F_q} \cap \Lambda. \]
Let $\beta \colon \F_{q^n}^n \times \F_{q^n}^n \rightarrow \F_{q^n}$ be the standard inner product, i.e. $\beta(({\bf x},{\bf y}))=\sum_{i=0}^{n-1} x_iy_i$ where ${\bf x}=(x_0,\ldots,x_{n-1})$ and ${\bf y}=(y_0,\ldots,y_{n-1})$.
Also, the restriction of $\beta$ over $W \times W$ is $\beta{\big|}_{W \times W}((x,x^q,\dots,x^{q^{n-1}}),(y,y^q,\dots,y^{q^{n-1}}))=\mathrm{Tr}_{q^n/q}(xy)$.
Furthermore, with respect to the orthogonal complement operation $\perp$ defined by $\beta$ we have that
\[ \Gamma^\perp : x_{j}= \sum_{\ell=0}^{r-1} g_{j,\ell} x_{t_\ell} \quad\quad j \notin \{t_0,\ldots,t_{r-1}\}. \]
Then the Delsarte dual $\bar U_\cC$ of $U_{\cC}$ is the $\F_q$-subspace $W+\Gamma^\perp$ of the quotient space $\F_{q^n}^n/\Gamma^\perp$ isomorphic to $U':=\langle W, \Gamma^\perp \rangle_{\F_q}\cap \Lambda'$,
where $\Lambda'$ is the $\F_{q^n}$-subspace of $\F_{q^n}^n$ of dimension $n-r$ represented by the following equations
\[ \Lambda' : x_{t_0}=\ldots=x_{t_{r-1}}=0. \]
By identifying $\Lambda'$ with $\F_{q^n}^{n-r}$, direct computations show that $U'$ can be seen as the $\F_q$-subspace $U_{\cC^\perp}=\{(h_0'(x),\ldots, h_{n-r-1}'(x)) : x \in\F_{q^n}\}$ of dimension $n$ of $\F_{q^n}^{n-r}$, i.e. $U'=U_{\cC^\perp}$.
\end{proof}

\section{\texorpdfstring{Intersections of maximum $h$-scattered subspaces with hyperplanes}{Intersections of maximum h-scattered subspaces with hyperplanes}}
\label{sec:inter}

This section is devoted to prove 

\medskip
\noindent
{\bf Theorem \ref{maininter}}	{\it If $U$ is a maximum $h$-scattered $\F_q$-subspace of a vector space $V(r,q^n)$ of dimension $rn/(h+1)$, then for any $(r-1)$-dimensional $\F_{q^n}$-subspace $W$ of $V(r,q^n)$ we have	\[\frac{rn}{(h+1)}-n\leq \dim_{\F_q}(U \cap W) \leq \frac{rn}{(h+1)}-n+h.\]}

As we already mentioned, the theorem above is a generalization of \cite[Theorem 4.2]{BL2000}, which is the $h=1$ case of our result. 
In that paper, the number of hyperplanes meeting a $1$-scattered subspace of dimension $rn/2$ in a subspace of dimension $rn/2-n$ or $rn/2-n+1$ has been determined as well.
Subsequently to this paper, in \cite{ZiniZ} (see also \cite{NZ} for the $h=2$ case), such values have been determined for every $h$.

\subsection{Preliminaries on Gaussian binomial coefficients}

The Gaussian binomial coefficient $\qbin{n}{k}{q}$ is defined as the number of the $k$-dimensional subspaces of the $n$-dimensional vector space $\F_q^n$. Hence

\begin{equation}
\label{formulanuova}
\qbin{n}{k}{q}=\left\{\begin{array}{ll}
1 & \mbox{if } k=0\\
\frac{(1-q^n)(1-q^{n-1})\ldots(1-q^{n-k+1})}{(1-q^k)(1-q^{k-1})\ldots(1-q)} &  \mbox{if } 1\leq k\leq n\\
0 & \mbox{if } k>n.
\end{array}\right.
\end{equation}
Recall the following properties of the Gaussian binomial coefficients.
\begin{equation}
\label{twotoone}
\qbin{n}{k}{q}\qbin{k}{j}{q}=\qbin{n}{j}{q}\qbin{n-j}{k-j}{q},
\end{equation}

\begin{equation}
\label{twotoone2}
\qbin{n}{k}{q}=\qbin{n}{n-k}{q}.
\end{equation}

\begin{equation}
\label{identity}
\prod_{j=0}^{n-1}(1+q^jt)=\sum_{j=0}^n q^{j(j-1)/2}\qbin{n}{j}{q}t^j,
\end{equation}

\begin{definition} The $q$-Pochhammer symbol is defined as
	\[(a;q)_k=(1-a)(1-aq)\ldots(1-aq^{k-1}).\]
\end{definition}

\begin{theorem}[$q$-binomial theorem {\cite[pg. 25, Exercise 1.3 (i)]{klopedia}}]
	\begin{equation}
	\label{eqa}
	(ab;q)_n=\sum_{k=0}^n b^k \qbin{n}{k}{q}(a;q)_k (b;q)_{n-k},
	\end{equation}
	\begin{equation}
	\label{eqb}
	(ab;q)_n=\sum_{k=0}^n a^{n-k} \qbin{n}{k}{q}(a;q)_k (b;q)_{n-k}.
	\end{equation}
	
\end{theorem}
\begin{corollary}
	In \eqref{eqa} and \eqref{eqb} put $a=q^{-nr/s}$ and $b=q^{nr/s-n}$ to obtain
	\begin{equation}
	\label{comp1}
	(q^{-n};q)_s=\sum_{j=0}^{s}q^{j(nr/s-n)}\qbin{s}{j}{q}(q^{-nr/s};q)_j(q^{nr/s-n};q)_{s-j},
	\end{equation}
	\begin{equation}
	\label{comp2}
	(q^{-n};q)_s=q^{-nr}\sum_{j=0}^s q^{jnr/s}\qbin{s}{j}{q}(q^{-nr/s};q)_j(q^{nr/s-n};q)_{s-j},
	\end{equation}
	respectively.
\end{corollary}

The $l$-th elementary symmetric function of the variables $x_1,x_2,\ldots, x_n$ is the sum of all distinct monomials which can be formed by multiplying together $l$ distinct variables.

\begin{definition}
	Denote by $\sigma_{k,l}$ the $l$-th elementary symmetric polynomial in $k+1$ variables evaluated in $1,q,q^2,\ldots,q^k$.
\end{definition}

\begin{lemma}[{\cite[Proposition 6.7 (b)]{Cameron}}]
	\label{Cam}
	\[\sigma_{k,l}=q^{l(l-1)/2}\qbin{k+1}{l}{q}.\]
\end{lemma}

We will also need the following $q$-binomial inverse formula of Carlitz.

\begin{theorem}[{\cite[special case of Theorem 2, pg. 897 (4.2) and (4.3)]{Carlitz}}]
	\label{inverseqbinom}
	Suppose that $\{a_k\}_{k\ge 0}$ and $\{b_k\}_{k\geq 0}$ are two sequences of complex numbers. If $a_k=\sum_{j=0}^k (-1)^jq^{j(j-1)/2}\qbin{k}{j}{q}b_j$, then $b_k=\sum_{j=0}^k (-1)^jq^{j(j+1)/2-jk}\qbin{k}{j}{q}a_j$ and vice versa.
\end{theorem}

\subsection{Double counting}
Put $s=h+1\mid rn$ and let $U$ be an $rn/s$-dimensional $\F_q$-subspace of $V(r,q^n)$ such that for each $(s-1)$-dimensional $\F_{q^n}$-subspace $W$, we have $\dim_{\F_q}(W \cap U)\leq s-1$.

Let $h_i$ denote the number of $(r-1)$-dimensional $\F_{q^n}$-subspaces meeting $U$ in an $\F_q$-subspace of dimension $i$. It is easy to see that
\[h_i=0 \mbox{ for } i<\frac{rn}{s}-n.\]
In $\PG(V,\F_{q^n})=\PG(r-1,q^n)$, the integer $h_i$ coincides with the number of hyperplanes $\PG(W,\F_{q^n})$ such that $\dim_{\F_q}(W\cap U)=i$. Also, the number of hyperplanes is $(q^{rn}-1)/(q^n-1)$, which is the same as $\sum_i h_i$, thus

\begin{equation}
\label{st}
\sum_i h_i(q^n-1)=q^{rn}-1.
\end{equation}

\noindent For $k\in \{0,1,\ldots, s-1\}$ we can double count the set
\[\{(H,(P_1,P_2,\ldots,P_{k+1})) : \mbox{$H$ is a hyperplane},\  P_1,P_2,\ldots,P_{k+1} \in H\cap L_U\] \[\mbox{\  \  \  \ and } \la P_1,P_2,\ldots,P_{k+1} \ra\cong\PG(k,q)\}.\]
By Proposition \ref{scatole} this gives

\[\sum_i h_i \left(\frac{q^i-1}{q-1}\right)\left(\frac{q^i-q}{q-1}\right)\ldots \left(\frac{q^i-q^k}{q-1}\right)=\]
\[\left(\frac{q^{rn/s}-1}{q-1}\right)\left(\frac{q^{rn/s}-q}{q-1}\right)
\ldots\left(\frac{q^{rn/s}-q^k}{q-1}\right)\left(\frac{q^{(r-k-1)n}-1}{q^n-1}\right),\]

\noindent or equivalently

\begin{lemma}
	\label{kif}
	\[\sum_i h_i(q^n-1)(q^i-1)(q^i-q)(q^i-q^2)\ldots(q^i-q^k)=\]
	\[(q^{rn/s}-1)(q^{rn/s}-q)(q^{rn/s}-q^2)\ldots(q^{rn/s}-q^k)(q^{(r-k-1)n}-1).\]
	\qed
\end{lemma}

\noindent Our aim is to prove
\[A:=\sum_i h_i(q^n-1)(q^i-q^{n(r-s)/s})\ldots(q^i-q^{n(r-s)/s+s-1})=0.\]
This would clearly yield $h_i=0$, for $i>n(r-s)/s+s-1$, and hence Theorem \ref{maininter}.

\subsection{\texorpdfstring{Expressing $A$}{Expressing A}}

First for $k\in \{0,\ldots, s-1\}$ we will express
\[\alpha_k:=\sum_i h_i(q^n-1)q^{ki}.\]
Put $\beta_0:=\alpha_0=q^{rn}-1$ (cf. \eqref{st}), and
\[\beta_k:=\sum_i h_i(q^n-1)(q^i-1)(q^i-q)\ldots(q^i-q^{k-1}),\]
where the values of $\beta_k$ are known due to Lemma \ref{kif}.

\noindent Recall
\[\sigma_{k,l}=\sum_{0\leq i_1<\ldots<i_l\leq k}q^{i_1+\ldots+i_l}.\]
Then it is easy to see that
\[\alpha_k=\beta_k+\sum_{j=0}^{k-1}(-1)^{k-j-1}\alpha_j \sigma_{k-1,k-j},\]
and hence, using also Lemma \ref{Cam},
\[\beta_{k}=\sum_{j=0}^k (-1)^{k-j}q^{(k-j)(k-j-1)/2} \alpha_{j}\qbin{k}{k-j}{q},\]
or equivalently, by \eqref{twotoone2},
\[\beta_kq^{-k(k-1)/2}(-1)^k=\sum_{j=0}^k q^{j(j+1)/2-jk}\qbin{k}{j}{q}(-1)^{j}\alpha_j.\]
Then Theorem \ref{inverseqbinom} applied to the sequences $\{a_k=\alpha_k\}_k$ and $\{b_k=\beta_kq^{-k(k-1)/2}(-1)^k\}_k$ gives

\begin{equation}
\label{kifejez}
\alpha_k=\sum_{j=0}^k(-1)^j q^{j(j-1)/2}\qbin{k}{j}{q}\beta_j q^{-j(j-1)/2}(-1)^j=\sum_{j=0}^k \qbin{k}{j}{q}\beta_j.
\end{equation}

\noindent It is easy to see that
\[
A=\sum_{j=0}^{s}(-1)^{s-j}\alpha_j q^{(s-j)n(r-s)/s}\sigma_{s-1,s-j}
\]
and hence by Lemma \ref{Cam}
\begin{equation}
\label{kifejez2}
A=\sum_{j=0}^{s}(-1)^{s-j}\alpha_j q^{(s-j)n(r-s)/s+(s-j)(s-j-1)/2}\qbin{s}{s-j}{q}.
\end{equation}

\noindent By Lemma \ref{kif} we have
\[\beta_{k}=(q^{(r-k)n}-1)\prod_{j=0}^{k-1}(q^{rn/s}-q^j)=\]
\[(q^{(r-k)n}-1)q^{k(k-1)/2}(-1)^{k}\prod_{j=0}^{k-1}(1-q^{rn/s-j}).\]
By \eqref{identity} with $t=-q^{rn/s-k+1}$
\[\prod_{j=0}^{k-1}(1-q^{rn/s-j})=\prod_{j=0}^{k-1}(1-q^{rn/s-k+1}q^j)=
\sum_{j=0}^{k} q^{j(j-1)/2+rnj/s-(k-1)j}\qbin{k}{j}{q}(-1)^j,\]
thus

\[\beta_k=\sum_{j=0}^{k}(q^{(r-k)n}-1)q^{k(k-1)/2} q^{j(j-1)/2+rnj/s-(k-1)j}\qbin{k}{j}{q}(-1)^{j+k}=\]
\[\sum_{j=0}^{k}(q^{(r-k)n}-1)q^{(k-j)(k-j-1)/2+rnj/s}\qbin{k}{j}{q}(-1)^{j+k}=\]
\[\sum_{t=0}^{k}(q^{(r-k)n}-1)q^{t(t-1)/2+rn(k-t)/s}\qbin{k}{t}{q}(-1)^{t}.\]

\noindent Hence by \eqref{kifejez} and \eqref{kifejez2}

\[A=\sum_{k=0}^{s}\sum_{j=0}^k
\sum_{t=0}^{j}(q^{(r-j)n}-1)q^{t(t-1)/2+rn(j-t)/s+(s-k)n(r-s)/s+(s-k)(s-k-1)/2}\qbin{s}{k}{q}\qbin{k}{j}{q}\qbin{j}{t}{q}(-1)^{t+k+s}.\]

\subsection{\texorpdfstring{Proof of $A=0$}{Proof of A=0}}

Since $q$-binomial coefficients out of range are defined as zero, cf.\,\eqref{formulanuova}, it is enough to prove that the following expression is zero:

\[\sum_{k=0}^{s}\sum_{j=0}^{s}
\sum_{t=0}^{s}(q^{rn-jn}-1)q^{rnj/s-rnt/s+(s-k)n(r-s)/s+\frac{1}{2} (s-k-1) (s-k)+\frac{1}{2} (t-1)t}\qbin{s}{k}{q}\qbin{k}{j}{q}\qbin{j}{t}{q}(-1)^{t+k}.\]

\noindent It is clearly equivalent to prove $a_s=b_s$, where

\[a_s=\sum_{j=0}^{s}q^{nr-nj}\sum_{k=0}^{s}
\sum_{t=0}^{s}q^{rnj/s-rnt/s+(s-k)n(r-s)/s+\frac{1}{2} (s-k-1) (s-k)+\frac{1}{2} (t-1)
	t}\qbin{s}{k}{q}\qbin{k}{j}{q}\qbin{j}{t}{q}(-1)^{t+k},\]

\[b_s=\sum_{j=0}^{s}\sum_{k=0}^{s}
\sum_{t=0}^{s}q^{rnj/s-rnt/s+(s-k)n(r-s)/s+\frac{1}{2} (s-k-1) (s-k)+\frac{1}{2} (t-1)
	t}\qbin{s}{k}{q}\qbin{k}{j}{q}\qbin{j}{t}{q}(-1)^{t+k}.\]

\begin{proposition}
	\[a_s=q^{nr}(-1)^s (q^{-n};q)_s.\]
\end{proposition}
\begin{proof}
	Clearly, it is enough to prove
	\[(-1)^s(q^{-n};q)_s=\]
	\[\sum_{j=0}^{s}q^{-nj}\sum_{k=0}^{s}(-1)^k\qbin{s}{k}{q}\qbin{k}{j}{q}
	q^{nrj/s+(s-k)n(r-s)/s+\frac{1}{2} (-k+s-1) (s-k)}
	\sum_{t=0}^{s}q^{-nrt/s+\frac{1}{2} (t-1)
		t}\qbin{j}{t}{q}(-1)^{t},\]
	where by \eqref{identity}
	\[\sum_{t=0}^{s}q^{-nrt/s+\frac{1}{2} (t-1)
		t}\qbin{j}{t}{q}(-1)^{t}=(q^{-nr/s};q)_j,\]
	thus the triple sum can be reduced to
	\[\sum_{j=0}^{s}q^{nrj/s-nj}(q^{-nr/s};q)_j\sum_{k=0}^{s}(-1)^k\qbin{s}{k}{q}\qbin{k}{j}{q}
	q^{(s-k)n(r-s)/s+\frac{1}{2} (-k+s-1) (s-k)}.\]
	By \eqref{twotoone} and \eqref{twotoone2} this can be written as
	\begin{equation}\label{formnew}
	\sum_{j=0}^{s}q^{nrj/s-nj}(q^{-nr/s};q)_j\qbin{s}{j}{q}\sum_{k=0}^{s}(-1)^k\qbin{s-j}{s-k}{q}
	q^{(s-k)n(r-s)/s+\frac{1}{2} (-k+s-1) (s-k)},
	\end{equation}
	where again by \eqref{identity}
	\[\sum_{k=0}^{s}(-1)^k\qbin{s-j}{s-k}{q}
	q^{(s-k)n(r-s)/s+\frac{1}{2} (-k+s-1) (s-k)}=\]
	\[(-1)^s\sum_{z=0}^{s}(-1)^{z}\qbin{s-j}{z}{q}q^{zn(r-s)/s+z(z-1)/2}=\]
	\[(-1)^s(q^{nr/s-n};q)_{s-j}.\]
	By \eqref{formnew} we have
	\[(-1)^s\sum_{j=0}^{s}q^{j(nr/s-n)}\qbin{s}{j}{q}(q^{-nr/s};q)_j(q^{nr/s-n};q)_{s-j},\]
	which by \eqref{comp1} equals $(-1)^s(q^{-n};q)_s$.
\end{proof}

\begin{proposition}
	\[b_s=q^{nr}(-1)^s(q^{-n};q)_s.\]
\end{proposition}
\begin{proof}
	As before, $b_s$ can be written as
	\[q^{nr}(-1)^s\sum_{j=0}^{s}q^{jnr/s-nr}\qbin{s}{j}{q}(q^{-nr/s};q)_j(q^{nr/s-n};q)_{s-j}.\]
	Then the assertion follows from \eqref{comp2}.
\end{proof}

\noindent Bence Csajb\'ok\\
MTA--ELTE Geometric and Algebraic Combinatorics Research Group\\
ELTE E\"otv\"os Lor\'and University, Budapest, Hungary\\
Department of Geometry\\
1117 Budapest, P\'azm\'any P.\ stny.\ 1/C, Hungary\\
{{\em csajbokb@cs.elte.hu}}

\bigskip

\noindent Giuseppe Marino\\
Dipartimento di Matematica e Applicazioni ``Renato Caccioppoli"\\
Università degli Studi di Napoli ``Federico II",\\
Via Cintia, Monte S.Angelo I-80126 Napoli, Italy\\
{\em giuseppe.marino@unina.it}

\bigskip

\noindent Olga Polverino and Ferdinando Zullo\\
Dipartimento di Matematica e Fisica,\\
Universit\`a degli Studi della Campania ``Luigi Vanvitelli'',\\
I--\,81100 Caserta, Italy\\
{{\em olga.polverino@unicampania.it}, \\{\em ferdinando.zullo@unicampania.it}}

\end{document}